\documentclass{article}

\usepackage{amsmath}
\usepackage{amsthm}
\usepackage{amsfonts}
\usepackage{amsbsy}
\usepackage{amssymb}
\newtheorem{theorem}{Theorem}
\newtheorem{proposition}{Proposition}
\newtheorem{corollary}{Corollary}

\newtheorem{lemma}{Lemma}

\newcommand{\R}{{\mathbb R}}
\newcommand{\Z}{{\mathbb Z}}

\newcommand{\C}{{\mathbb C}}

\newcommand{\set}[2]{ \left\{ #1 \ \left| \ #2 \right. \right\} }

\newtheorem*{theoremst}{Theorem 6.4}

\newcommand{\B}{{\mathbb B}}
\newcommand{\FS}{C^m(\B)}

\title{Scalar oscillatory integrals in smooth spaces of homogeneous type}
\author{Philip T. Gressman}

\begin{document}
\maketitle

\begin{abstract}
We consider a generalization of the notion of spaces of homogeneous type, inspired by recent work of Street \cite{street2011} on the multi-parameter Carnot-Carath\'{e}odory geometry, which imbues such spaces with differentiability structure.  The setting allows one to formulate estimates for scalar oscillatory integrals on these spaces which are uniform and respect the underlying geometry of both the space and the phase function.  As a corollary we obtain a generalization of a theorem of Bruna, Nagel, and Wainger \cite{bnw1988} on the asymptotic behavior of scalar oscillatory integrals with smooth, convex phase of finite type.
\end{abstract}

Given a manifold $\Omega$ and a measure of smooth density $\mu$, a frequent problem in analysis is to establish an estimate for scalar oscillatory integrals of the form
\begin{equation} \int_{\Omega} e^{if} \psi d \mu \label{estimatethis} \end{equation}
where the phase $f$ is real-valued and the amplitude $\psi$ is supported on a set of finite measure.  We would like the estimates to be uniform in $f$ and $\psi$ and to effectively reduce the problem of estimating this integral to a sublevel set problem for the gradient of $f$ (since the method of stationary phase dictates that there will be substantial cancellation away from critical points of $f$).  Ideally, this objective should be accomplished in a geometrically-invariant way if at all possible, although to date this has proven to be a difficult task to accomplish, especially in dimension greater than one.  We would also like to assume as little as possible about $\psi$, as it is generally regarded to be of secondary importance when contrasted with the phase.

A tremendous amount of work has already been devoted to understanding problems in the form \eqref{estimatethis} and related objects in higher dimensions.  If one is willing to compromise somewhat on uniformity requirements (and, for example, restrict attention phases $f$ which are scalar multiples of a single fixed phase function $\Phi$ and simple perturbations of such phases), methods based on resolution of singularities provide extremely powerful tools for understanding \eqref{estimatethis}.  The seminal result in this direction is due to Var{\v{c}}enko \cite{varcenko1976}; the history of this field is lengthy and we will not try to summarize it here, but we will note that the emphasis of some of the most recent work, due to Collins, Greenleaf, and Pramanik \cite{cgp2010} as well as Greenblatt \cite{greenblatt2008}, has been to produce resolution of singularities algorithms which are much more concrete (and more easily applied by non-specialists) than were previously available; these new algorithms are also able to handle phase functions which fell outside the scope of earlier work for technical reasons.  Answering questions of uniformity by means of resolution of singularities is still generally a difficult task.   It is also worth noting that, somewhat paradoxically, resolution of singularities methods tend to encounter added difficulties when the decay rate of scalar oscillatory integrals is relatively high.
Work in the complementary direction, emphasizing uniformity of some form or another (and sacrificing on sharpness of the estimates if it becomes necessary) also abounds: see, for example,
Carbery, Christ, and Wright \cite{ccw1999};  Carbery and Wright \cite{cw2002};
Greenblatt \cite{greenblatt2011};  Iosevich \cite{iosevich1999}; 
Ikromov, Kempe, and M\"{u}ller \cite{ikm2007}; Karpushkin \cite{karpushkin1986}; Phong and Stein \cite{ps1998};  Phong, Stein, and Sturm \cite{pss1999,pss2001}; Phong and Sturm \cite{ps2000};
Pramanik and Yang \cite{py2004}; 
Rogers \cite{rogers2005}; Seeger \cite{seeger1998}; and the author \cite{gressman2008,gressman2010II}.  The goals, ideas, and methods to be found in these results are numerous and diverse.

If we further narrow attention to uniform estimates of \eqref{estimatethis} which display some degree of geometric invariance, one is quickly left with a rather short list of known results.  Of these we highlight the work of Bruna, Nagel, and Wainger \cite{bnw1988}, in which they succeed in estimating the decay of the Fourier transform of a smooth, convex surface of finite type in terms of the volume growth rates of surface caps.
The goal of the present paper is to establish a theorem in the same spirit as the theorem of Bruna, Nagel, and Wainger while removing the convexity assumption (and, in some sense, the finite type assumption).  It turns out that there is a natural way to expand the familiar definition of spaces of homogeneous type, originally due to Coifman and Weiss \cite{cw1971}, to include a compatible differential structure.  This was already observed in some sense by Bruna, Nagel, and Wainger; however, the important point here is that the construction can be completely divorced from a fine-structure analysis of the phase and does not, for example, require convexity.  Using this new structure, we establish a natural estimate of \eqref{estimatethis} which is both uniform and firmly tied to the underlying geometry of the manifold and the phase.  The new smoothness hypotheses added to spaces of homogeneous type are intuitive and, for the most part, have already been shown to hold in many of the familiar cases.  In particular, the smoothness hypotheses are satisfied in Carnot-Carath\'{e}odory geometries, as is shown in the thread of papers beginning with Nagel, Stein, and Wainger \cite{nsw1985}; including Tao and Wright \cite{tw2003}; and culminating with Street \cite{street2011}.

We begin the definitions by assuming only that $\Omega$ is a topological space and $\mu$ is a Borel measure on $\Omega$.  The familiar axioms associated with spaces of homogeneous type begin with a family of balls $B_j(x) \subset \Omega$.  One's intuition should be that the balls are geometrically nice sets containing $x$, but note that we will explicitly avoid the assumption they are open.   Here $x$ may be any point in $\Omega$, and in this paper the index $j$ will be contained in $\Z^d$; this should be thought of as corresponding to Street's ``multiparameter'' setting of Carnot-Carath\'{e}odory geometry (although the main theorems of this paper have interesting new consequences even in the case $d=1$ corresponding to single parameter geometry). For technical reasons, we will allow the possibility that $B_j(x) = \emptyset$ for certain values of the parameter $j$; we will say that such balls do not exist or are not defined.  In all other cases, (i.e., when $B_j(x) \neq \emptyset$), it will be required that $x \in B_j(x)$. Regarding the index $j \in \Z^d$ (referred to as the scale of the ball), we will use the standard notations that $j' \leq j$ when the corresponding inequality holds for each coordinate of $j'$ and $j$ in $\Z^d$, and we will define $|j-j'|$ to be the $\ell^\infty$-norm on $\Z^d$.  We will also identify the integers $\Z$ with the diagonal subset of $\Z^d$, i.e., $n = (n,\ldots,n)$.

The (mostly) familiar assumptions regarding the geometry of these balls are recorded here.   We suppose that for some open $\Omega_0 \subset \Omega$, we have the following:
\begin{enumerate}
\item[i.] (Compatibility) If $B_{j'}(x') \cap B_j(x) \neq \emptyset$ for some $x, x' \in \Omega_0$ and $j,j' \in \Z^d$, then $B_{j-1}(x')$ exists.
%\item[ii.] (Nesting) If $B_j(x) \neq \emptyset$ for some $x \in \Omega_0$ and $j \in \Z^d$, then $j' \leq j-1$ implies $B_{j'}(x) \subset B_{j}(x)$.
\item[ii.] (Engulfing) Suppose $B_{j}(x)$ is defined for some $x \in \Omega_0$ and $j \in \Z^d$. If there exists $x' \in \Omega_0$ and $j' \leq j-1$ such that $B_{j-1}(x) \cap B_{j'}(x') \neq \emptyset$, then $B_{j'}(x') \subset B_{j}(x)$.
\item[iii.] (Weak Doubling) There is a finite constant $C$ such that any distinct points $x_1,\ldots,x_N$ of $\Omega_0$ and any index $j \in \Z^d$ with $B_{j}(x_k) \cap B_{j}(x_l) = \emptyset$ for all $k \neq l$ have the property that at most $C$ of these points satisfy $B_{j+1}(x) \cap B_{j+1}(x_k) \neq \emptyset$ for any fixed $x \in \Omega_0$.
\end{enumerate}
Next we impose additional smoothness structure.  For each nonempty ball $B_j(x)$, we assume that there is a homeomorphism $\Phi_{j,x} : \B^{d_x} \rightarrow B_j(x)$ which maps $0$ to $x$, where $\B^{d_x}$ is the open Euclidean unit ball in dimension $d_x$.  The dimension $d_x$ may depend on $x$, but we assume that any two balls which intersect have the same dimension.  We will also assume that the supremum over $x$ of $d_x$ is finite; it will be referred to as the dimension when no confusion will arise.  We will abuse notation and use $\B$ to refer to the unit ball in Euclidean space of appropriate dimension (depending on context).  In a nutshell, we will assume that these homeomorphisms are smooth with respect to each other when compared on two comparable balls.   Specifically we assume:
\begin{enumerate}
\item[iv.] (Smooth Nesting) For some universal $c < 1$, \[ \sup_{t \in \B} |\Phi_{j,x}^{-1} \circ \Phi_{j-1,x}(t)| \leq c. \]
\item[v.] (Smooth Engulfing) If $B_{j}(x) \cap B_{j'}(x') \neq \emptyset$, then there is a constant $C_{|j-j'|}$ depending only on $|j-j'|$, dimension, and $m$ such that
\begin{equation} \left| \partial^{\alpha}_t \left[ \Phi_{j,x}^{-1} \circ \Phi_{j',x'}(t) \right] \right| \leq C_{|j-j'|}  \label{smeng}
\end{equation}
uniformly as $t$ ranges over $\Phi_{j',x'}^{-1}(B_{j}(x) \cap B_{j'}(x'))$ and $\alpha$ ranges over all multiindices of order at most $m$.  
\end{enumerate}
%We note in particular that the last assumption is the main reason for allowing failure at large scales (since, in some sense, we cannot expect large balls to be uniformly close to flat).
Under these assumptions, it is possible to quantify the smoothness of a function $f$ at any particular point $x$ and any given scale $j$.  We specifically define
\begin{equation} |d^k_x f|_j := \sup_{1 \leq |\alpha| \leq k} \left| \left. \partial^{\alpha}_t \left[ f \circ \Phi_{j,x} (t) \right] \right|_{t=0} \right| \label{pointsmdef} \end{equation}
for any $k = 1,\ldots,m$.  We also denote $|d^1_x f|_j$ by $|d_x f|_j$ when no confusion will arise.  Under the assumptions above, the quantity $|d_x f|_j$ satisfies a sort of weak differential invariance property: namely, that compositions of $f$ with ``tame'' diffeomorphisms (measured by composing with $\Phi_{j,x}$ for each $j$ and $x$) will preserve the magnitude of $d_x f$ at scale $j$ up to a bounded factor.  That this sort of weak invariance is the best that may be hoped for can be seen by considering the effect of a rough (in this case, $C^1$) change of variables in the integral \eqref{estimatethis}.
If such a rough transformation is made, the value of the integral will remain constant, but the method of stationary phase will fail for technical reasons.  Thus the only way to deduce decay for such an integral would to understand the very precise coincidence of the {\it irregularity} of both $f$ and $\psi$.  Typically for applications one would like to assume as little as possible about the function $\psi$.  Thus we are necessarily constrained to consider only the effect of tame diffeomorphisms.

The final set of assumptions we make concern the (Borel) measure $\mu$ and its regularity with respect to the balls $B_j(x)$.  To that end,  a set $L \subset \Omega$ is called a leaf when it has the following properties:
\begin{enumerate}
\item Every ball $B_j(x)$ with $x \in \Omega_0$ which intersects $L$ is contained in $L$.
\item Every ball contained in $L$ is relatively open in $L$.
\item $L$ has a countable dense subset.
\end{enumerate}
We will assume that $\Omega$ is equipped with a measure $\mu$ which may be ``factored'' onto the leaves in the following sense: 
\begin{enumerate}
\item[vi.] (Regularity of Measure) There is a collection of leaves $\mathcal F$ (i.e., a foliation) with measure $\mu_{\mathcal F}$ and Borel measures $\mu_L$ on each leaf $L \in {\mathcal F}$ such that
\[ \int_{\Omega} f d \mu = \int_{\mathcal F} \left[ \int_L f d \mu_L \right] d \mu_{\mathcal F}(L) \]
for all Borel measurable functions $f$ on $\Omega$.  Each measure $\mu_L$ should have smooth density, meaning that for any leaf $L$ and any ball $B_j(x) \subset L$ with $x \in \Omega_0$ and $\mu_L(B_j(x)) \neq 0$, there is a nonvanishing function $J_{j,x}$ such that 
\begin{equation} \frac{1}{\mu_L(B_j(x))} \int_{B_j(x)} f d \mu_L = \int_{\B} f \circ \Phi_{j,x}(t) J_{j,x}(t) dt \label{jacobian} \end{equation}
with $||J_{j,x}||_{\FS} \leq C$ and $\inf_{t \in \B} J_{j,x}(t) \geq c$ uniformly in $j$, $x$, and $L$.
\end{enumerate}
We must make the technical assumption that the balls $B_j(x)$ are Borel measurable and the maps
\[ x \mapsto \Phi_{j,x}(u) \]
for fixed $j$ and $u$ are defined on a Borel measurable set and are Borel measurable functions there.
\begin{theorem}
Assume that (i) through (vi) hold.  Fix any $\epsilon \in (0,1)$, and suppose $m \geq 2$.  Let $E \subset \Omega_0$ consist of those points where $d_x f \neq 0$, and suppose $R : E \rightarrow \Z^d$ some Borel measurable function such that $B_{R(x)+1}(x)$ is well-defined for each $x \in E$ and each of the following conditions hold: \label{maintheorem}
\begin{align}
 B_{R(x)}(x) \cap B_{R(y)}(y) & \neq \emptyset \Rightarrow |R(x) - R(y)| \lesssim 1, \label{lipschitz} \\
 |d^{m}_x f|_{R(x)} & \lesssim \sum_{k=1}^{m-1} \epsilon^{k-m} |d_x^k f |_{R(x)}, \label{highderiv} \\
 \sup_{y \in B_{R(x)}(x)} \epsilon |d_y f|_{R(y)} & \lesssim 1 +  \inf_{y \in B_{R(x)}(x)} \epsilon |d_y f|_{R(y)}  \label{firstderiv}
\end{align}
with implicit constants uniform with respect to $x,y \in E$ and $\epsilon$.
Then for any smooth, bounded $\psi$ whose support has finite measure in $\Omega_0$, there is another function $\psi_m$ such that
\[ \int_{\Omega_0} e^{if} \psi d \mu = \int_{\Omega_0} e^{if} \psi_m d \mu \]
and
\begin{equation} |\psi_m(x)| \lesssim \frac{ \sum_{k=0}^{m-1} \epsilon^k |d^{k}_x \psi(x)|_{R(x)}}{(1 + \epsilon |d_x f(x)|_{R(x)})^{m-1}}. \label{mainineq}
\end{equation}
\end{theorem}

%We will say that two systems of balls $\{(B_j(x), \Phi_{j,x})\}$ and $\{(B'_j, \Phi'_{j,x})\}$ are comparable when the maps $\Phi_{j,x}' \circ \Phi_{j,x}^{-1}$ and their inverses are uniformly in $C^m$ at all points where this composition makes sense (and implicitly we assume that the dimensions of the corresponding balls are the same in the two systems), and  $B_{j-2N}(x) \subset B'_{j-N}(x) \subset B_{j}(x)$ for each $j$ (with appropriate modification when one or more balls in this chain is empty) and some fixed $N$.
%
%The compatibility condition is merely a minimal requirement on the scales at which a ball $B_{j}(x)$ must be defined: if it is possible to make sense of the ball $B_j(x)$ and it obeys (ii) through (vii), then it should be possible to define balls $B_{j-1}(x')$ at all points $x'$ which are ``near'' $x$ when viewed from the scale $j$.
%
%The first assumption merely dictates that the balls should be defined uniformly (i.e., near a well-defined ball of scale $j$, it will always be possible to define balls of scale at least $j-1$).
%We call the last assumption the weak doubling hypothesis since it can be easily established in the presence of a doubling measure inequality (also exploiting engulfing) if one knows {\it a priori} that the measures of balls are never zero or infinite.

Informally, the theorem establishes a more geometric version of the method of stationary phase, meaning that the usual gradient factor appears in the denominator, but only when measured at the correct scale at each point.  The condition \eqref{lipschitz} can be thought of as a sort of Lipschitz condition for the scales (meaning that if you measure on a very fine scale at $x$, you must also measure at fine scales when relatively near $x$).
The inequality \eqref{highderiv} plays the role of a finite type condition, but it should not be understood literally as such, since it only holds on balls away from the set where $d_x f = 0$.  In particular, in most applications an $\epsilon$ is guaranteed to exist satisfying \eqref{highderiv} uniformly provided only that the ball $B_{R(x)}(x)$ is sufficiently small at each point.   Elementary examples show that this can, in fact, be achieved even in some cases when $f$ is not of finite type.

One interesting corollary of theorem \ref{maintheorem} concerns the original inequality of Bruna, Nagel, and Wainger concerning scalar oscillatory integrals with convex phases.  Using theorem \ref{maintheorem}, we may extend this earlier estimate in two ways.  The first is that we will have a slightly stronger sort of uniformity than was originally available---we will explicitly identify quantities that determine the values of the implicit constants and we will not, for example, assume that the domain of integration is compact.  We will also explicitly identify the number of derivatives necessary for estimates to hold (so one need not actually assume that the convex phase is $C^\infty$). The second (and perhaps more interesting) extension is that the finite-type assumption will be replaced by a strictly weaker one which in some cases includes convex functions which are not finite type (and even phases not strictly convex).

Let us begin by identifying the substitute notion of regularity we will employ.  Let $f \in C^m([0,T])$ for some $m \geq 2$ be a nonnegative convex function with $f(0) = 0$ and $f'(0) \geq 0$.  We say that $f$ is tame on $[0,T]$ to order $m$ when there is a constant $C < \infty$ such that for each $k=2,\ldots,m$ and each $t \in [0,T]$ we have
\begin{equation} |f(t)|^{k-1} |f^{(k)}(t)| \leq C |f'(t)|^k. \label{tame} \end{equation}
Although the tameness inequality may at first seem unusual, it is a natural way to measure regularity of $f$ in a way that is invariant under rescaling the magnitude of $f$ and rescaling the time parameter $t$:  for any positive $\alpha,\beta$, the same constant $C$ will be possible for the corresponding inequalities applied to the function $g(t) := \alpha f (\beta^{-1} t)$ on the interval $[0, \beta T]$.  In fact, we may reinterpret \eqref{tame} as follows: if the horizontal and vertical axes of the graph of $f$ are rescaled in such a way that the rescaled function $g$ satisfies $g(t_0) = g'(t_0) = 1$ at some point $t_0$, then the higher derivatives of $g$ at $t_0$ (and, hence, the coefficients of the Taylor polynomial at $t_0$) will be controlled by the constant $C$.

Two important features of \eqref{tame} should be highlighted at this point.  The first is that for polynomial $f$, the constant $C$ appearing in \eqref{tame} can be taken to depend only on the degree of $f$ (a fact which shall be explicitly addressed in section \ref{bnwsec}).  The second is that \eqref{tame} does not imply that $f$ is finite type (or even strictly convex).  For example, a trivial induction argument establishes that $f(t) := e^{-(1 + \alpha^{-1}) t^{-\alpha}}$ is tame to any finite order on $[0,1]$ for any positive $\alpha$.

Definitions now established, we have the following corollary of theorem \ref{maintheorem}:
\begin{corollary}
Suppose $f$ is a convex function on an open convex set $\Omega \subset \R^d$ containing the origin.  Assume that $f$ and its gradient vanish at the origin.  If $f$ is uniformly radially tame to order $m$ on $U$ (meaning that $f$ is tame with the same constant on all rays beginning at the origin), then for any compactly supported $\psi \in C^m(\Omega)$ and any real $\lambda$, \label{bnwcor}
\begin{equation} \left| \int_{\Omega} e^{i \lambda f(x)} (f(x))^\ell \psi(x) dx \right| \lesssim \int_{\Omega} \frac{|f(x)|^\ell \sum_{k=0}^{m-1} |(x \cdot \nabla)^k \psi(x)|  }{(1 +  |\lambda f(x)|)^{m-1}} dx \label{sublev} \end{equation}
with implicit constant depending only on $m$, $d$, the nonnegative integer $\ell$, and the constant of uniform radial tameness.
\end{corollary}

The proof of the corollary is contained in section \ref{bnwsec} along with two different propositions which establish stable estimates for the constant of uniform radial tameness (the first proposition dealing with polynomials as promised, and the second establishing finiteness of the constant for convex functions of finite type in the sense of \cite{bnw1988}).  The final result in that section demonstrates how estimates of the sort established by Bruna, Nagel, and Wainger in terms of volumes of caps can be deduced in the standard way from \eqref{sublev}.

The structure of the rest of this paper is as follows.  Section \ref{initial} establishes a few immediate observations and is then devoted to a brief study of functions of finite type on Euclidean balls.  Though relatively elementary, this section includes a concrete construction on the Euclidean ball of analogues of the classical Littlewood-Paley projections (which eliminates problems near the boundary that occur in more straightforward approaches) which may be of independent interest.  Theorem \ref{finitetype} falling in this section can be thought of as a sharp characterization of functions of finite type on the ball.  By sharp we mean that its main hypothesis \eqref{sharpcond} is implied (with a slightly worse constant) by and substantially weaker than the main conclusion \eqref{bigfinish}.  Section \ref{theproof} contains the main covering lemma and the proof of theorem \ref{maintheorem}.  Finally, section \ref{therest} is devoted to three topics.  The first regards the choice of scale function $R$ in theorem \ref{maintheorem}.  We show that, in many contexts, there is a natural choice for $R$ and that, roughly speaking, it can be thought of as simply the largest scale on which the magnitude of the gradient looks roughly constant.  The second topic in section \ref{therest} is the proof of corollary \ref{bnwcor} and an analysis of uniform radial tameness.  Finally we remind the readers of the relevant definitions, theorems, and inequalities from Street \cite{street2011} and present the final theorem, theorem \ref{ccworks}, which illustrates how the structures behind theorem \ref{maintheorem} arise naturally in the context of Carnot-Carath\'{e}odory geometry.

\section{Initial steps} \label{initial}
\subsection{Basic observations}

When $B_{j'}(x') \subset B_j(x)$, the mapping $\Phi_{j,x}^{-1} \circ \Phi_{j',x'}$ is a smooth map from the ball $\B$ into itself.
The smooth engulfing property implies (via the Leibniz and chain rules) that
\begin{equation} || f \circ \Phi_{j,x}^{-1} \circ \Phi_{j',x'} ||_{\FS} \lesssim (1+ C_{|j-j'|})^m ||f||_{\FS} \label{smengplus}
\end{equation}
with an implied constant depending only on dimension and $m$ (where $C_{|j-j'|}$ is the same constant appearing in \eqref{smeng}).
More generally, suppose $f$ is a smooth function on $\Omega$.  Fix a point $x_0 \in \Omega_0$.  For any ball $B_j(x)$ containing $x_0$ with $x$ also in $\Omega_0$, we may measure the smoothness of $f$ at $x_0$ by means of the formula
\[ | d^k_{x_0} f|_{B_j(x)} :=  \sup_{1 \leq |\alpha| \leq k} \left| \left. \partial_t^{\alpha} \left[ f \circ \Phi_{j,x} (t)\right] \right|_{t = \Phi_{j,x}^{-1}(x_0)} \right|, \]
i.e., by taking the usual mixed partials on the Euclidean ball $\B$ and evaluating at the appropriate point (note that the definition \eqref{pointsmdef} corresponds to the case $x=x_0$, i.e., $|d_x^k f|_j = |d^k_x f|_{B_j(x)}$).  Although the magnitude of $d f_{x_0}$ depends on the choice of the ball $B_j(x)$, smooth engulfing will dictate the comparability condition
\begin{equation}
(1 + C_{|j-j'|})^{-m} |d^k_{x_0} f|_{B_{j'}(x')}  \lesssim |d^k_{x_0} f|_{B_{j}(x)} \lesssim 
(1 + C_{|j-j'|})^{m} |d^k_{x_0} f|_{B_{j'}(x')} \label{comparable}
\end{equation}
whenever $|d^k_{x_0} f|_{B_{j}(x)}$ and $|d^k_{x_0} f|_{B_{j'}(x')}$ are well-defined (with implicit constant depending only on the usual suspects of dimension and $m$).

One final elementary observation is that when $x \in \Omega_0$ and $B_j(x) \neq \emptyset$, there is a smooth, nonnegative function $\eta_{j,x}$ supported on $B_j(x)$ (meaning it is identically zero outside this ball) and bounded above by one which is identically one on $B_{j-1}(x)$ and has
\begin{equation} ||\eta_{j,x} \circ \Phi_{j',x'}||_{\FS} \lesssim (1+C_{|j-j'|})^m \label{smoothdouble2} \end{equation}
for any ball $B_{j'}(x')$.  The function $\eta_{j,x}$ is {\it not} necessarily smooth on $\Omega$, but is nevertheless Borel measurable.  These assertions follow immediately from smooth doubling: choose any $C^\infty$ function $\eta$ on the ball $\B$ which is identically zero outside the ball of radius $(1+c)/2$ centered at the origin, identically one on the ball of radius $c$ centered at the origin, and maps into $[0,1]$.  We define $\eta_{j,x}(x_0) := \eta \circ \Phi_{j,x}^{-1}(x_0)$ when $x_0 \in B_j(x)$ and $\eta_{j,x}(x_0) = 0$ when $x' \not \in B_j(x)$.  By construction we have uniform control on $|d^k_{x_0} \eta_{j,x}|_{B_{j'}(x')}$ (as it depends only on $\eta$) when $x_0 \in B_{j}(x)$.  In all other cases, $\eta_{j,x}(x_0)$ and all its derivatives will vanish identically.  The function $\eta_{j,x}$ is Borel measurable because it is continuous on the Borel set $B_{j}(x)$.

\subsection{An aside on functions of finite type}
\label{ftsec}
An important notion frequently tied to the estimation of oscillatory integrals is that of functions of finite type.  While there are many variations and generalizations of this notion appearing in the literature, at its core, a function of finite type is one which is nearly polynomial.  The principal benefit of restricting attention to these functions is that they satisfy an inequality of the form
\[ \sup_{|\omega| < 1} |\partial^{\alpha} f(\omega)| \leq C_{\alpha,f} \sup_{|\omega| < 1} |f(\omega)| \]
for some constant $C_{\alpha,f}$ which depends on $f$ only in a relatively tame way (i.e., in terms of its $C^m$ norm and lower bounds for the nonvanishing derivative, etc.).  This inequality may be thought of as bounding the high frequency components of $f$ by the low frequencies (and so we have the heuristic that functions which are locally like polynomials are also locally like slowly-varying complex exponential functions).  A closely related problem is to determine the maximal $\epsilon$ such that 
\[ \sup_{|\omega| < 1} |\partial^{\alpha} f(\omega)| \leq \epsilon^{-|\alpha|} \sup_{|\omega| < 1} |f(\omega)| \]
holds for a given $f$ and a certain range of derivatives $\alpha$.  In this subsection we record a theorem which gives a sharp answer to this question of determining the optimal $\epsilon$ up to a universal constant depending only on dimension and the degree of smoothness of $f$.  It gives a quantitative analogue of the finite-type condition which is useful in the proof of theorem \ref{maintheorem} and is hopefully interesting in its own right.  The proof can be accomplished by adaptation of the usual real variable Littlewood-Paley methods to the setting of the unit ball.
\begin{theorem}
Suppose $f \in C^m(\B)$ satisfies the inequality
\begin{equation}
\mathop{\sup_{|\omega| < 1}}_{|\alpha| = m} \epsilon^{|\alpha|} |\partial^{\alpha} f(\omega)| \leq \mathop{\sup_{|\omega| < 1}}_{\ell \leq |\alpha| < m} \epsilon^{|\alpha|} |\partial^{\alpha} f(\omega)| \label{sharpcond}
\end{equation}
for some positive $\epsilon \leq 1$ and some integer $\ell < m$.
Then there is an implicit constant depending only on dimension and $m$ such that
\begin{equation} \sup_{|\omega| < 1} |\partial^{\alpha} f(\omega)| \lesssim \epsilon^{-|\alpha|} \mathop{\sup_{|\omega| < 1}}_{|\beta| = \ell} \epsilon^{|\beta|} |\partial^{\beta} f(\omega)| \label{bigfinish}
\end{equation}
for all multiindices $\ell \leq |\alpha| \leq m$. \label{finitetype}
\end{theorem}

It should be remarked that \eqref{bigfinish} clearly implies \eqref{sharpcond} for some $\epsilon'$ differing from $\epsilon$ by a uniform constant.  As such we essentially have a characterization of \eqref{bigfinish}.  The theorem, however, is far from tautological, since \eqref{sharpcond} is easily verified while \eqref{bigfinish} is not.  In particular \eqref{sharpcond} is immediately true with $\epsilon = 1$ when $f$ is a polynomial of degree at most $m-1$.  Likewise it is easy to see that an acceptable $\epsilon$ satisfying \eqref{sharpcond} may be given proportional to ratio $||f||_{C^{m-1}(\B)} / ||f||_{C^m(\B)}$ (which is how the theorem is typically applied to functions of finite type).  A third example, relevant to the Carnot-Carath\'{e}odory geometry, is given after the proof.

\begin{proof}[Proof of theorem \ref{finitetype}]
Let $\varphi$ be a smooth function compactly supported in $\B$.  We suppose that $\varphi$ is even and satisfies the moment conditions
\[ \int_{\B} \varphi(x) dx = 1 \mbox{ and }  \int_{\B} x^{\alpha} \varphi(x) dx = 0 \mbox{ when } 1 \leq |\alpha| \leq m-1. \]
Using this $\varphi$, we consider the following family of operators for $j$ a nonnegative integer:
\begin{align*}
 P_jf (x,h)  & := 2^{dj} \int f(x-z) \sum_{k=0}^{m} 2^{kj} \frac{((h \cdot \nabla)^k \varphi)(2^j z)}{k!}  dz.
\end{align*}
This is well-defined for bounded, continuous functions on the ball provided that $|x| \leq 1 - 2^{-j}$.  Note that the dependence on $h$ is polynomial.  In particular, with $x$ fixed, the $h$ dependence is exactly the degree $m$ Taylor polynomial of $P_j f(x+h,0)$ at $h=0$, so $P_j f(x,h) = f(x+h)$ when $f$ is any polynomial of degree at most $m$.  We may therefore reasonably think of $P_j f(x,h)$ as an analogue of the Littlewood-Paley projection of $f$ onto frequencies $2^j$ and below.   Just as with Littlewood-Paley projections, we have uniform estimates
\begin{equation} \sup_{|x| \leq 1-2^{-j}} |\partial^{\beta}_h P_j f(x,h)| \lesssim 2^{j |\beta|} (1 + (2^j |h|)^{m-|\beta|}) \sup_{|\omega| < 1} |f(\omega)| \label{lpbdd} \end{equation}
with implied constant which is independent of $f$, $h$, and $j$ (but may depend on the multiindices, $m$, and dimension).

If $f \in C^{\ell}(\B)$ for some $\ell \leq m$, we may integrate by parts to conclude
\begin{equation}
\begin{split}
 P_j f(x,h) = & 2^{dj} \int \varphi(2^j(x-z)) \sum_{k=0}^{\ell-1} \frac{((h \cdot \nabla)^k f)(z)}{k!} dz \\
 & + 2^{dj} \int ((h \cdot \nabla)^{\ell} f)(x-z) \sum_{k = 0}^{m - \ell} 2^{kj} \frac{((h \cdot \nabla)^k \varphi)(2^j z)}{(k + \ell)!}  dz.
\end{split} \label{lpdif}
\end{equation}
In particular, if $|\beta| \geq \ell$, then the partial derivative $\partial^{\beta}_h$ kills the first term on the right-hand side of \eqref{lpdif}.  Thus we have a slight improvement of \eqref{lpbdd}:
\begin{equation} \mathop{\sup_{|x| \leq 1-2^{-j}}}_{|h| \leq 2^{-j}} |\partial^{\beta}_h P_j f(x,h)| \lesssim 2^{j (|\beta|-\ell)} \mathop{\sup_{|\omega| < 1}}_{|\gamma| = \ell} |\partial^{\gamma} f(\omega)|. \label{lpbdd2} \end{equation}
Returning to \eqref{lpdif}, changing variables in the first integral on the right-hand side gives
\[ \int \varphi(2^j(x-z)) \sum_{k=0}^{\ell-1} \frac{((h \cdot \nabla)^k f)(z)}{k!} dz =  \int \varphi(2^j z) \sum_{k=0}^{\ell-1} \frac{((h \cdot \nabla)^k f)(x-z)}{k!} dz. \]
By Taylor's Theorem, we have the inequalities
\begin{align*}  \left| \sum_{k=0}^{\ell-1} \frac{((h \cdot \nabla)^k f)(x-z)}{k!} - f(x-z+h) \right| & \leq \sup_{|\omega| < 1} \frac{|(h \cdot \nabla)^{\ell} f(\omega)|}{\ell!}, \\
 \left| \sum_{k=0}^{\ell-1} \frac{((-z \cdot \nabla)^k f)(x+h)}{k!} - f(x-z+h) \right| & \leq \sup_{|\omega| < 1} \frac{|(z \cdot \nabla)^{\ell} f(\omega)|}{\ell!}, 
\end{align*}
provided that $x-z$ and $x-z+h$ belong to the unit ball in the first case and $x+h$, $x-z+h$ belong to the ball in the second case.  In particular, the second inequality contains a polynomial in $z$ of degree less than $m$; if multiplied by $\varphi(2^{j}z)$ and integrated in $z$, all terms but the constant term will cancel.  We conclude from these estimates and \eqref{lpdif} that the quantity $|P_j f(x,h) - f(x+h)|$ is dominated by

\begin{align*}
& 2^{dj} \int |\varphi(2^j z)| \left[ \sup_{|\omega| < 1} \frac{|(h \cdot \nabla)^{\ell} f(\omega)|}{\ell!} + \sup_{|\omega| < 1} \frac{|(z \cdot \nabla)^{\ell} f(\omega)|}{\ell!} \right] dz \\
 & + 2^{dj} \int \left[ \sup_{|\omega| < 1} \frac{|(h \cdot \nabla)^{\ell} f(\omega)|}{\ell!} \right] \sum_{k=0}^{m-\ell} 2^{kj} \left| \frac{ (( h \cdot \nabla)^k \varphi)(2^j z)}{(k +  \ell )!} \right| dz.
 \end{align*}
%By standard techniques this sum may itself be dominated by
%\begin{align*}
% \frac{1}{(\ell+1)!} & \left( \sum_{|\beta| = \ell+1} \sup_{|\omega| < 1} |\partial^\beta f(\omega)|^2 \right)^{\frac{1}{2}} \left[ ||\varphi||_1 (|y|^{\ell+1} + 2^{-j (\ell+1)}) \vphantom{         \left( \sum_{|\beta| = k} ||\partial^\beta \varphi||_1^2 \right)^{\frac{1}{2}}  } \right. \\
% &  \ \left. + |y|^{\ell+1}  \sum_{k=0}^{m - \ell - 1} 2^{kj} |y|^k \frac{(\ell+1)!}{(k + \ell + 1)!} \left( \sum_{|\beta| = k} ||\partial^\beta \varphi||_1^2 \right)^{\frac{1}{2}} \right].
%\end{align*}
%Suppose from now on that $|h| \leq 2^{-j} K$ for some fixed constant $K$.   
From here it is easy to see that there must be an implicit constant depending only on $\varphi$, dimension, and $m$ so that $|x| \leq 1 - 2^{-j}$ and $|x+h| < 1$ imply
\begin{equation} |P_j f(x,h) - f(x+h) | \lesssim 2^{-j\ell} (1 + (2^j |h|)^m) \mathop{\sup_{|\omega| < 1}}_{|\beta| = \ell} |\partial^{\beta} f(\omega)|. \label{lpest1}
\end{equation}
It should also be noted that the inequality above will also hold trivially when $\ell = 0$ by virtue of \eqref{lpbdd}.

For fixed $f$, let $g_j^{x} (y) := P_j f(x,y-x)$.
Because the dependence of $g_j^x(y)$ on $y$ is that of a polynomial of degree at most $m$, we have that
\[ P_{j-1} g_j^{x} (x+\delta,h-\delta) = g^{x}(x+ h) = P_j f(x,h). \]
Comparison to $P_{j-1} f(x+\delta, h+\delta)$ yields
\[ P_j f(x,h) = P_{j-1} f (x+\delta,h-\delta) + P_{j-1} (g^x_j - f) (x+\delta,h-\delta) \]
provided $|x+\delta| \leq 1 - 2^{-j+1}$ and $|x| \leq 1 - 2^{-j}$.  Now every point in the ball of radius $1-2^{-j}$ is within distance $2^{-j}$ of  a point in the ball of radius $1 - 2^{-j+1}$.  Let $E_j$ be the set of pairs $(x,h)$ where $|x| \leq 1 - 2^{-j}$, $|x+h| < 1$, and $|h| < 2^{-j+1}$; we have
\begin{align*}
\mathop{\sup_{|x| \leq 1 - 2^{-j}}}_{|h| \leq 2^{-j}}&  |\partial^{\alpha}_h P_j f(x,h)| - \mathop{\sup_{|x| \leq 1 - 2^{-j+1}}}_{|h| \leq 2^{-j+1}}  |\partial^{\alpha}_h P_{j-1} f(x,h)| \\ &  \lesssim  2^{j |\alpha|} \sup_{(x,h) \in E_j} |P_j f(x,h) - f(x+h)| \end{align*}
where the implicit constant comes from \eqref{lpbdd}.  In light of \eqref{lpbdd2}, summing $j$ from $N+1$ to infinity and using the fact that $\partial^{\alpha}_h P_j f(x,h)$ at $h=0$ tends to $\partial^{\alpha} f(x)$ as $j \rightarrow \infty$ when $f \in C^{m}(\B)$ and $|\alpha| \leq m$ (shown by integration by parts as usual) gives
\begin{equation} \begin{split} \sup_{|\omega| < 1} |\partial^{\alpha} f(\omega)| \lesssim  & 2^{(|\alpha| - \ell)N} \mathop{\sup_{|\omega| < 1}}_{|\beta| = \ell} |\partial^{\beta} f(\omega)| \\ & + \sum_{j=N+1}^\infty 2^{|\alpha| j} \sup_{(x,y) \in E_j}  |P_j f(x,y) - f(x+y)|. \end{split} \label{lpest3}
\end{equation}

Now suppose that $f$ satisfies the inequality \eqref{sharpcond}
for some $\epsilon \leq 1$.  Choose $N$ so that $2^N = \epsilon^{-1} \delta^{-1}$ for some $\delta < 1$ to be chosen momentarily.  Applying the inequalities \eqref{lpest1} (with $\ell$ replaced by $m$) and \eqref{sharpcond} to the sum on the right-hand side of \eqref{lpest3} gives
\begin{align*}
 \sup_{|\omega| < 1} |\partial^{\alpha} f(\omega)| \lesssim & \  (\epsilon \delta)^{\ell -|\alpha|} \mathop{\sup_{|\omega| < 1}}_{|\beta| = \ell} |\partial^{\beta} f(\omega)| + (\epsilon \delta)^{m - |\alpha|} \mathop{\sup_{|\omega| < 1}}_{|\beta| = m} |\partial^{\beta} f(\omega)| \\
\lesssim & \ \epsilon^{- |\alpha|}  \left[ \delta^{\ell -|\alpha|} \mathop{\sup_{|\omega| < 1}}_{|\beta| = \ell} \epsilon^{|\beta|} |\partial^{\beta} f(\omega)| + \delta^{m - |\alpha|} \mathop{\sup_{|\omega| < 1}}_{|\beta| < m} \epsilon^{|\beta|} |\partial^{\alpha} f(\omega)| \right]
\end{align*}
(for $|\alpha| = \ell$ this inequality is trivially true).  If we choose $\delta$ small enough that $\delta$ times the implicit constant is between $\frac{1}{4}$ and $\frac{1}{2}$ (this will always be possible since we may increase the magnitude of the implicit constant as necessary), we may take a supremum over all such $\alpha$ with $\ell \leq |\alpha| < m$ and conclude
\[ \mathop{\sup_{|\omega| < 1}}_{\ell \leq |\alpha| < m} \epsilon^{|\alpha|} |\partial^{\alpha} f(\omega)| \lesssim \mathop{\sup_{|\omega| < 1}}_{|\beta| = \ell} \epsilon^{|\beta|} |\partial^{\beta} f(\omega)|. \]
Combining this inequality with \eqref{sharpcond} itself handles the case of multiindices $\alpha$ with $|\alpha| = m$.
\end{proof}

We now return to the issue of establishing the hypothesis \eqref{sharpcond} in a manner relevant to the Carnot-Carath\'{e}odory geometry.   Suppose $Y_1,\ldots,Y_d$ are $C^m$ vector fields on $\B^d$.  We will say that a function $f$ is of polynomial type with respect to $Y_1,\ldots,Y_d$ if 
\[ Y_{i_1} \cdots Y_{i_M} f \equiv 0 \]
for some fixed $M$ and all choices $(i_1,\ldots,i_M) \in \{1,\ldots,d\}^M$.  Any such function will automatically satisfy \eqref{sharpcond} for some $\epsilon$ which depends only on the $C^m$ norms of the vector fields $Y_i$, the infimum of the absolute value of the determinant $\det(Y_1,\ldots,Y_d)$, and the constant $M$.  This is because we may write
\[ \frac{\partial}{\partial x_j} = \sum_{i=1}^d c_{ji}(x) Y_i \]
for some functions $c_{ij} \in C^m(\B)$ whose norms depend only on the norms of the $Y_i$'s and the lower bound from the determinant; this fact is easily seen from Cramer's rule.  In particular, we see that the mixed partial $\partial^{\beta} f$ may be written as a $C^m$ linear combination of derivatives $Y_{i_1} \cdots Y_{i_k}$ for $k \leq M$.  Since these expressions vanish identically when $k = M$; we see that the order $M$ partial derivatives of $f$ are a smooth linear combination of all lower order partial derivatives, and the coefficients do not depend on $f$.

\section{Proof of the main theorem} \label{theproof}

\subsection{Covering lemma}

The benefit of the smoothness assumptions is that it allows the passage from a Vitali- or Besicovitch-type covering lemma to a smooth partition of unity whose derivatives are well-controlled.

\begin{lemma}
Fix some positive integer $N$ and some open subset $E \subset \Omega_0$. \label{partlemma}
 Suppose $R : E \rightarrow \Z$ is a bounded function satisfies the properties that $B_{R(x)+1}(x)$ is well-defined for each $x \in E$ and that $$B_{R(x)} (x) \cap B_{R(x')}(x') \neq \emptyset \Rightarrow |R(x) - R(x')| < N.$$
 Then there is a special collection $G$ of points $x \in E$ and nonnegative functions $\eta_x$ for each $x \in G$ which satisfy a number of properties.  First, $\eta_x$  identically zero outside $B_{R(x)}(x)$.  Next, at every $y \in E$, there is a ball $B_j(y)$ for which there are at most a uniformly bounded number of points $x \in G$ at which $B_{R(x)}(x) \cap B_j(y) \neq \emptyset$.  The functions $\eta_x$ are uniformly smooth in the sense that $x \in G$, $||\eta_x \circ \Phi_{R(x),x}||_{\FS} \lesssim 1$.  Finally, $\sum_{x \in G} \eta_x = 1$ on $E$.  \end{lemma}
\begin{proof}
Let $E_j \subset E$ be the set of points $x$ for which $R(x) = j$.   For each $j$,  let $G_{j} \subset E_j$ be any maximal collection of points $x$ such that $B_{j-2}(x) \cap B_{j-2}(x') = \emptyset$ for any two $x,x' \in G_j$ (all these balls must exist by virtue of the compatibility condition).  
 First we note that the union of the balls $B_{j-1}(x)$ over all $x \in G_j$ will cover $E_j$: for any $y \in E_j$, $B_{j-2}(y) \cap B_{j-2}(x) \neq \emptyset$ for some $x \in G_j$ (if not, $y$ itself could be added to $G_j$ to contradict maximality).  The engulfing property guarantees that $B_{j-2}(y) \subset B_{j-1}(x)$.
Next observe that the family of balls $B_j(x)$ for $x \in G_j$ are locally finite in the following sense: for any $x \in E$, let $S$ be the set of centers $x' \in G_j$ such that $B_j(x) \cap B_{j}(x') \neq \emptyset$ (if $B_j(x)$ is not defined, the set $S$ will be trivial).  This set $S$ must necessarily be finite with uniformly bounded cardinality.  To see this, let $S_1$ be any maximal subset of $S$ of disjoint balls at scale $j-1$, i.e., $B_{j-1}(x') \cap B_{j-1}(x'') = \emptyset$ for any $x',x'' \in S_1$.  In general, let $S_k$ be a maximal subset of $S \setminus \ \bigcup_{l=1}^{k-1} S_l$ of disjoint balls at scale $j-1$.  The weak doubling property dictates that $S_{C+2}$ is empty for some universal $C$, since maximality dictates that $x' \in S_{C+1}$ implies that $B_{j-1}(x') \cap B_{j-1}(x_k) \neq \emptyset$ for $k=1,\ldots,C+1$ and some $x_k \in S_k$ (which cannot happen because the balls are all mutually disjoint at scale $j-2$).  By weak doubling again, the number of points $x' \in S_k$ at which $B_j(x) \cap B_j(x') \neq \emptyset$ is also at most $C$ for any fixed $k$.  Thus the total number of indices in $S$ which produce balls at scale $j$ meeting $B_j(x)$ is at most $C(C+1)$.
We may strengthen this result by taking a union over scales.  For any point $x \in E$, the condition $|R(x) - R(x')| < N$ when $B_{R(x)}(x) \cap B_{R(x')}(x') \neq \emptyset$ implies that there are boundedly many indices $j'$ for which $B_{R(x)}(x)$ intersects a ball $B_{R(x')}(x')$ with $x' \in G_{j'}$.  If $j$ is any index for which is bounded above by $j'-1$ for each index $j'$ identified above as well as bounded above by $R(x)-1$ (at least one such index, e.g., $j = R(x)-N$, is always guaranteed to exist), then then the number of points $x' \in G_{j'}$ for which $B_{j}(x) \cap B_{R(x')}(x') \neq \emptyset$ will be at most $C(C+1)$ (because we will have in particular that $B_j(x) \subset B_{R(x')}(x)$ by engulfing).  Uniform boundedness on the cardinality of the possible values of $j'$ gives a uniform bound on the number of nontrivial intersections $B_{j}(x) \cap B_{R(x')}(x') \neq \emptyset$ when $x'$ is allowed to range over all of $G := \bigcup_j G_j$.

As explained in \eqref{smoothdouble2}, there is a natural choice of a smooth function subordinate to $B_{j}(x)$ for each $x \in G_j$, namely the function denoted $\eta_{j,x}$.
Furthermore we have $||\eta_{j,x} \circ \Phi_{R(y),y}||_{\FS} \lesssim 1$ for any $y \in E$ simply because $\eta_{j,x} \circ \Phi_{R(y),y}$ will be identically zero unless $B_{R(x)}(x) \cap B_{R(y)}(y) \neq \emptyset$, in which case we already have the uniform bound $|R(x) - R(y)| < N$ on the indices $R(x)$ and $R(y)$ (which finishes the job when combined with \eqref{smoothdouble2}).
Now by the distributive law we have
\begin{align*}
1 & =  \prod_{x \in G} \left[ (1 - \eta_{R(x),x}) + \eta_{R(x),x} \right] \\ &  = \mathop{\sum_{S \subset G}}_{\#S < \infty} \left( \prod_{x \in S} \eta_{R(x),x} \right) \left( \prod_{y \in  G \setminus S} (1 - \eta_{R(y),y}) \right) 
\end{align*}
since on any ball $B_{R(z)-N}(z)$ with $z \in E$ all but boundedly many choices of $x$ will have $(1 - \eta_{R(x),x})$ which is identically one on this ball and $\eta_{R(x),x}$ identically zero.

Given a finite subset $S \subset G$, let $M(S)$ be the subset of $S$ drawn from $G_j$'s with maximal indices:  specifically, for each $x \in S$, $x$ belongs to $G_j$ for a unique index $j$.  We will take $x \in M(S)$ if and only if $S \cap G_k = \emptyset$ for all $k > j$.  We will say that two finite subsets $S,S'$ are equivalent when $M(S) = M(S')$.  On any equivalence class $\mathcal S$, $M(\mathcal S)$ is well-defined (since $M(S)$ is constant for all representatives $S$).  If we call the collection of equivalence classes $\mathcal E$, we have
\begin{align*}
1 & = \prod_{x \in G} (1 - \eta_{R(x),x}) + \sum_{{\mathcal S} \in {\mathcal E} \setminus \{\emptyset\}} \sum_{S \in {\mathcal S}} \left( \prod_{x \in S} \eta_{R(x),x} \right) \left( \prod_{y \in  G \setminus S} (1 - \eta_{R(y),y}) \right)
\end{align*}
(where we identify $\emptyset \in {\mathcal E}$ to be the equivalence class of the empty set).
Since $\eta_{j,x}$ is identically one on $B_{j-1}(x)$ and the union of the balls $B_{R(x)-1}(x)$ over $x \in G$ covers $E$, the first product on the right-hand side will be identically zero on $E$.  

For a fixed equivalence class $\mathcal S$, let $I_{\mathcal S}^0$ be the indices $j$ such that $G_j \cap M({\mathcal S}) \neq \emptyset$, and let $I_{\mathcal S}^{-}$ be the indices $j$ such that $j < k$ for some $k \in I_{\mathcal S}^0$.  By definition, the representatives $S$ of the equivalence class $S$ are precisely given by the union of $M(\mathcal S)$ with any fixed subset of $\bigcup_{j \in I_{\mathcal S}^{-}} G_j$.  In particular, if $I_{\mathcal S}^+$ is the complement of $I^0_{\mathcal S} \cup I^{-}_{\mathcal S}$, then every representative $S$ has $S \cap G_j = \emptyset$ when $j \in I^+_{\mathcal S}$.  Consequently, the distributive law guarantees that
\begin{align*}
\sum_{S \in {\mathcal S}} \left( \prod_{x \in S} \eta_{R(x),x} \right) & \left( \prod_{y \in  G \setminus S} (1 - \eta_{R(y),y}) \right) = \\ &  \left( \prod_{x \in M({\mathcal S})} \eta_{R(x),x} \right) \left( \prod_{y \in  G_{\mathcal S}^+ \setminus M({\mathcal S})} (1 - \eta_{R(y),y}) \right) 
\end{align*}
if we define $G^0_{\mathcal S} := \bigcup_{j \in I^0_{\mathcal S}} G_j$ and $G^+_{\mathcal S} := \bigcup_{j \in I^+_{\mathcal S}} G_j \cup \bigcup_{j \in I^0_{\mathcal S}} G_j$.
Moreover, $B_{R(x)}(x) \cap B_{R(y)}(y) = \emptyset$ when $|R(x) - R(y)| \geq N$, so the above formula remains true when $G_{\mathcal S}^+$ is replaced by the (substantially smaller) union of those $G_j$ for which $j \in I_{\mathcal S}^+ \cup I_{\mathcal S}^0$ {\it and} $|j - k| < N$ for {\it all} $k \in I^0_{\mathcal S}$.  Since this set has uniformly bounded cardinality, we may conclude that, for any $z \in E$, there is some ball $B_{j'}(z)$ on which the set of equivalence classes $\mathcal S$ giving rise to a nontrivial (i.e., not identically zero) product has uniformly bounded cardinality.  This coupled with smooth comparability guarantees that the composition of any such product with any $\Phi_{R(z),z}$ will have uniformly bounded norm in $\FS$ (since at each point of the ball the function is locally a product of bounded cardinality, and the definition of $R$ implies that each ball appearing in the product will have index uniformly near $R(z)$ if the product isn't simply identically zero).  Finally, if we set
\[ \eta_x := \sum_{{\mathcal S} \ : \ x \in M(\mathcal S)} \frac{1}{\#M({\mathcal S)}} \left( \prod_{z \in M({\mathcal S})} \eta_{R(z),z} \right) \left( \prod_{y \in  G_{\mathcal S}^+ \setminus M({\mathcal S})} (1 - \eta_{R(y),y}) \right) \]
(which again is well-defined, since only boundedly many choices of $\mathcal S$ for fixed $x$ will give a sum which is not identically zero on $B_{R(x)}(x)$), we have that
\[ 1 = \sum_{x \in G} \eta_x \]
on $E$ and $||\eta_x \circ \Phi_{R(x),x}||_{\FS}$ is uniformly bounded.  That $\eta_x$ is supported on $B_{R(x)}(x)$ follows because the same is true of $\eta_{R(x),x}$ itself.
\end{proof}

It is important to note that the set $G$ constructed by the lemma is not necessarily countable and so we have not technically constructed a parition of unity in the usual sense.  However, we will have that $G \cap L$ is countable for any leaf $L$.  This is because the sets $E \cap B_{R(x)}(x)$ are contained and open in $L$ for each $x$.  Moreover, no point $y \in L$ is contained in more than boundedly many of these sets $E \cap B_{R(x)}(x)$ for $x \in E$.  Since $L$ has a countable dense subset, the pigeonhole principle demands that there can be only countably many $x$ for which $E \cap B_{R(x)}(x) \cap L$ is nonempty for any particular leaf $L$ (as $x$ ranges over all of $G$).  In the event that the function $R$ satisfies the condition \eqref{firstderiv} from the statement of the main theorem, it is easy to see that
\[ \sum_{x \in G \cap L} \psi_x = \chi_{L \cap E} \]
(where $\chi$ represents the characteristic function) for the simple reason that the complement of $E$ will be the set where $d f$ vanishes, and \eqref{firstderiv} guarantees that $B_{R(x)}(x)$ does not contain any such points whenever $x \in E$.  By the factorization property of $\mu$, then, we have that
\begin{equation} \int_{\Omega} e^{if} \psi d \mu = \int_{\Omega \setminus E} e^{if} \psi d \mu + \int_{\mathcal F} \left( \sum_{x \in G \cap L} \int_{B_{R(x)}(x)} e^{if} \psi \eta_x d \mu_L \right) d \mu_{\mathcal F}(L) \label{lebesgue} \end{equation}
(the assumption that $\psi$ is bounded and supported on a set of finite measure in $\Omega_0$ guarantees that we trivially have dominated convergence on almost every leaf $L$).   

\subsection{Integral estimates and conclusion}

We have thus successfully reduced the problem to the very classical one of a scalar oscillatory integral on a Euclidean ball.  The main result we will use in this context is contained in the following lemma:
\begin{lemma}
Suppose $f \in \FS$ with nonvanishing gradient.  Then for each $k=1,\ldots,m-1$, there exist functions $F_\beta$ for all multiindices $|\beta| \leq k$ such that \label{balllemma}
\begin{equation} \int_{\B} e^{i f} \psi dx = \int_{\B} e^{if} \left[ \sum_{|\beta| \leq k} \partial^{\beta} \psi(x) F_{\beta}(x) \right] dx \label{sharpibp} \end{equation}
for any compactly supported $\psi \in \FS$.  The functions $F_{\beta}$ satisfy the inequalities
\begin{equation}
\left| F_\beta(x) \right| \lesssim \frac{(\omega(x))^{|\beta|} }{(\omega(x) |\nabla f(x)|)^{k}} \label{sharpjunk}
\end{equation}
for any nonnegative function $\omega$ satisfying 
\[ \frac{1}{\omega(x)} \geq \sup_{2 \leq |\gamma| \leq k+1} \left( \frac{|\partial^{\gamma} f(x)|}{|\nabla f(x)|} \right)^{\frac{1}{|\gamma|-1}}. \]
The implicit constant in \eqref{sharpjunk} depend only on the dimension and $m$.
\end{lemma}
\begin{proof}
The proof is by the time-honored method of integration-by-parts.  Let us suppose that $u : \R^d \rightarrow \C$ is smooth and homogeneous of degree $a$.  We integrate by parts as follows:
\begin{align*}
 \int_{\B} e^{i f(x)} \partial^{\beta} \psi(x) & u(\nabla f(x)) \prod_{j=1}^n \partial^{\gamma_j} f(x) dx  \\ & = \int_{\B} \frac{ \nabla f(x)}{i |\nabla f(x)|^2} \cdot  \nabla ( e^{i f(x)}) \partial^{\beta} \psi(x) u(\nabla f(x)) \prod_{j=1}^n \partial^{\gamma_j} f(x) dx \\
& =  i \int_{\B} e^{if} \nabla \cdot \left( \partial^{\beta} \psi(x) \frac{u(\nabla f(x)) \nabla f(x)}{|\nabla f(x)|^2} \prod_{j=1}^n \partial^{\gamma_j} f(x) \right) dx.
\end{align*}
Now there are three cases to consider:  if the derivatives present in the divergence fall on $\partial^{\beta} \psi$, then we may write the resulting terms as
\[ i \int_{\B} e^{if} \sum_{\ell=1}^d \partial^{\beta+e_\ell} \psi(x) u_\ell(\nabla f(x)) \prod_{j=1}^n \partial^{\gamma_j} f(x) \]
where $e_\ell$ is the multiindex corresponding to differentation with respect to $x_\ell$ and $u_\ell(y) := u(y) y_\ell |y|^{-2}$, which will be smooth and homogeneous of degree $\alpha-1$.  When the divergence falls on the $\nabla f$ terms, we get
\[ i \int_{\B} e^{if} \partial^{\beta} \psi(x) \sum_{\ell = 1}^d \sum_{k=1}^d \left( \partial_k u_\ell \right)(\nabla f(x)) \partial^2_{x_\ell x_k} f(x) \prod_{j=1}^n \partial^{\gamma_j} f(x) dx. \]
When expanded, we find that each term remains of the same form with $u$ being replaced by $\partial_{k} u_\ell$ (which is smooth and homogeneous of degree $a-2$) and the cardinality of the product of higher derivatives increasing by one while the total number of derivatives present in the product increases by two.  Finally, if the divergence falls on one of the higher derivatives of $f$, the only effect is to increase the order of differentation on that term by $1$.  By induction, we conclude that
\[ \int_{\B} e^{i f} \psi dx = \sum_{|\beta| \leq k} \int e^{if} \partial^{\beta} \psi F_\beta dx \]
as desired (in the base case, $u(\nabla f) \equiv 1$ is homogeneous of degree $0$), where each $F_\beta$ is a linear combination (with universal coefficents depending only on the dimension, $k$, and $\beta$) of terms of the form
\[ u(\nabla f(x)) \prod_{j=1}^n \partial^{\gamma_j} f(x). \]
The index $n$ can be taken less than or equal to $k$ (the case equalling zero meaning no higher derivatives are present).  Here $u$ will be smooth and homogeneous of degree no greater than $-k$.  More precisely, an analysis of the three cases above yields that
\[ |\beta| + \deg u + \sum_{j=1}^n |\gamma_j| = 0, \ \deg u + n = -k, \mbox{ and } k \leq |\beta| + \sum_{j=1}^n |\gamma_j| \leq 2k. \]
It is equally elementary to see that $|\gamma_j|$ can be at most $k+1$ for any $k$.  In particular, given the definition of $\omega(x)$, we have that
\[ \prod_{j=1}^n |\partial^{\gamma_j} f(x)| \leq \prod_{j=1}^n |\nabla f(x)| (\omega(x))^{-|\gamma_j|+1}  = ( \omega(x) |\nabla f(x)| )^{n} (\omega(x))^{-\sum_{j=1}^n |\gamma_j|}. \]
If we set $s = \sum_{j=1}^d |\gamma_j|$, then we may conclude that
\begin{align*} |F_\beta(x)| & \lesssim \max_{s + |\beta| =k,\ldots,2k} |\nabla f(x)|^{-s-|\beta|} ( \omega(x) |\nabla f(x)| )^{-k + |\beta| + s} (\omega(x))^{-s} \\
& \lesssim \frac{(\omega(x))^{|\beta|} }{(\omega(x) |\nabla f(x)|)^{k}},
\end{align*}
which finishes the lemma.
\end{proof}

Now we return to the expression \eqref{lebesgue} and apply \eqref{sharpibp}.  Specifically we have
\begin{align}
 \int_{B_{R(x)}(x)} e^{if} \psi \eta_x d \mu_L  & =
\mu_L(B_{R(x)}(x)) \int_{\B} e^{ i \tilde f(t)} \tilde \psi(t) dt \label{parallel}
\end{align}
where
\begin{align*}
\tilde f(t) & := f \circ \Phi_{R(x),x}(t), \\
\tilde \psi(t) & := \psi \circ \Phi_{R(x),x}(t) \eta_x \circ \Phi_{R(x),x}(t) J_{R(x),x}(t).
\end{align*}
By design, the product $\eta_x \circ \Phi_{R(x),x}(t) J_{R(x),x} (t)$ has uniformly bounded norm in $\FS$ and so may (in essence) be neglected.

Let $C$ be a constant to be chosen momentarily.  If $|d_y f|_{R(y)} \leq  C \epsilon^{-1} $ for all points $y \in B_{R(x)}(x)$, then observe that we have the trivial identity
\[ \int_{\B} e^{i \tilde f(t)} \tilde \psi(t) dt = \int_{\B} e^{i \tilde f(t)} \tilde \psi(t) dt \]
and the trivial inequality
\begin{equation} |\tilde \psi(t)| \lesssim \frac{|\psi \circ \Phi_{R(x),x}(t)|}{(1 + \epsilon |d_y f|_{R(y)})^{m-1}} \label{bound1} \end{equation}
for any $y \in B_{R(x)}(x)$ (with implicit constant depending on $C$).

If $|d_y f|_{R(y)} \geq C \epsilon^{-1}$ at some point $y \in B_{R(x)}(x)$, then the assumption \eqref{firstderiv} implies a uniform bound from below at every point in the ball when $C$ is chosen sufficiently large relative to the implicit constant in \eqref{firstderiv}.   In this case we will apply lemma \ref{balllemma} to the oscillatory integral.   Observe that when $y = \Phi_{R(x),x}(t)$, we have
\begin{equation} \sup_{1 \leq |\gamma| \leq k} \frac{ |\partial^{\gamma} \tilde f(t)|}{|\nabla f(t)|} \approx \frac{ |d_y^k f|_{B_{R(x)}(x)}}{|d_y f|_{B_{R(x)}(x)}} \label{ratio}
\end{equation}
(with universal constants depending only on dimension).  Since $y \in B_{R(x)}(x)$ and consequently $|R(x) - R(y)| < N$, we have by \eqref{comparable} that
\[ |d_y^k f|_{B_{R(x)}(x)} \approx |d^k_y f|_{B_{R(y)}(y)} = |d^k_y f|_{R(y)} \]
uniformly for any $k=1,\ldots,m$ (with implied but uninteresting dependence on $N$ as well as the usual constants).  
By \eqref{highderiv} and theorem \ref{finitetype}, now, we can conclude that the ratio \eqref{ratio} is bounded above uniformly in $y$ by $\epsilon^{-k}$ and thus the function $\omega$ from the lemma may be taken equal to some uniform constant times $\epsilon$.  We can conclude that
\[ \int_{\B} e^{i \tilde f(t)} \tilde \psi(t) dt = \int_{\B} e^{i \tilde f(t)} \tilde \psi_m(t) dt \]
and
\begin{equation} |\tilde \psi_m(t)| \lesssim \frac{\sum_{k=0}^{m-1} \epsilon^k |d^{k}_{y} \psi|_{R(y)}}{(1 + \epsilon |d_y f|)^{m-1}} \label{bound2}
\end{equation}
uniformly when $\Phi_{R(x),x}(y) = t$ (by virtue of \eqref{sharpjunk} and the fact that $|d_y f| \approx 1 + |d_y f|$ on this particular ball).  

Next take either \eqref{bound1} or \eqref{bound2} and transform back to the measure $\mu_L$:  in the case of \eqref{bound1}, we define the function $\psi_x$ by  $\psi_x := \psi \eta_x$
and in the case of \eqref{bound2} we take
\begin{equation} \psi_x := \left( \frac{\tilde \psi_m}{J_{R(x),x}} \right) \circ   \Phi_{R(x),x}^{-1}. \label{moveback} \end{equation}
Since $J^{-1}_{R(x),x}$ is uniformly bounded above, we may conclude that (in both cases) we have an equality
\[ \int e^{i f} \psi \eta_x d \mu_L = \int e^{i f} \psi_x d \mu_L \]
where $\psi_x$ is zero outside $B_{R(x)}(x)$ and
\[ |\psi_x(y)| \lesssim \frac{\sum_{k=0}^{m-1} \epsilon^k |d^{k}_{y} \psi|_{R(y)}}{(1 + \epsilon |d_y f|)^{m-1}} \] % \[ \frac{ |\psi(y)| + |d^{m-1}_y \psi|_{R(y)}}{(1 + |d_y f|_{R(y)})^{m-1}} \] 
when $y \in B_{R(x)}(x)$ (with constant uniform in $x$ and $y$).  We now invert the interchange of summation and integration found in \eqref{lebesgue}.  Because the balls chosen in lemma \ref{partlemma} were locally finite on $E$, the uniform bound we just established for $|\psi_x(y)|$ continues to hold when summed over $x \in G \cap L$ on any choice of leaf $L$.  This establishes the main theorem and, in particular, the desired inequality \ref{mainineq}.  We also note that everything here is Borel measurable so the inversion of \eqref{lebesgue} is justified.  This concludes the proof of theorem \ref{maintheorem}.

A side remark: It is possible to to remove the assumption that the Jacobian from \eqref{jacobian} is bounded below if one is willing to pay a price in terms of the amplitude $\psi$.  Instead of \eqref{moveback}, the equalities
\begin{align*}
 \int_{\B} e^{i \tilde f(t)} \tilde \psi_m(t) dt & = \int_{\B} J_{R(x),x}(s) ds \int_{\B} e^{i \tilde f(t)} \tilde \psi_m(t) dt  \\
 & = \int_{\B} e^{i \tilde f(s)} \left[ \int_{\B} e^{i (\tilde f(t) -\tilde f(s))} \tilde \psi_m(t) dt \right] J_{R(x),x}(s) ds \end{align*}
 mean that you could alternately define 
 \[ \psi_x(\cdot) := \int_{\B} e^{i (\tilde f(t) -\tilde f \circ \Phi_{R(x),x}^{-1}(\cdot))} \tilde \psi_m(t) dt.\]
You would no longer need the Jacobian to be bounded below, but you would pay for it by selecting an amplitude $\psi_x$ (and hence an amplitude $\psi_m$ in theorem \ref{maintheorem}) which is only bounded above by a sort of maximal function of the derivatives of the original phase as opposed to a simpler pointwise supremum of those derivatives.

\section{Applications and Extensions} \label{therest}

\subsection{A canonical construction of scale function $R(x)$.}

The selection of the scale function $R$ may, in the general case when scales are understood in the multiparameter sense, be more of an art than a science.  However, in the single scale case (corresponding to scales parametrized by the integers $\Z$), it is relatively easy to see that there is always, in some sense, a ``best'' choice of $R$.  Fix an $\epsilon$ and some implicit constants, and let $\mathcal I$ be the subcollection of balls $B_{j}(x)$ as $(j,x)$ ranges over pairs $ \Z \times \Omega$ for which 
\[ |d_x^m|_j \lesssim \sum_{k=1}^{m-1} \epsilon^{k-m} |d_x^k f|_j \mbox{ and } \sup_{y \in B_j(x)} |d_y f|_{j-1} \lesssim \inf_{y \in B_j(x)} |d_y|_{j-1}. \]
Let $\mathcal I$ denote this subset of $\Z \times \Omega$. 
 We will define $R_{\mathcal I}(x)$ to be the supremum over all those indices $j$ such that $B_{j+1}(x)$ exists and $B_{j'}(x') \in I$ for any ball $B_{j'}(x') \subset B_{j+1}(x)$ with $j' \leq j$.

The claim is that this mapping $R_{\mathcal I}$ satisfies the necessary regularity condition \eqref{lipschitz}.  Specifically, let $E$ be the set of points $x \in \Omega$ at which $R_{\mathcal I}(x)$ is well-defined and finite.  For any $x,x' \in E$, suppose that $B_{R_{\mathcal I}(x)}(x) \cap B_{R_{\mathcal I}(x')}(x') \neq \emptyset$.  By compatibility and nesting, $B_{R_{\mathcal I}(x)}(x')$ exists and is contained in $B_{R_{\mathcal I}+1}(x)$, and by the definition of $R_{\mathcal I}(x)$, it must therefore be the case that all balls of scale at most $R_{\mathcal I}(x)-1$ which are contained in $B_{R_{\mathcal I}(x)}(x')$ must also belong to $I$.  We thus conclude that $R_{\mathcal I}(x') \geq R_{\mathcal I}(x) - 1$.  By symmetry we have $|R_{\mathcal I}(x') - R_{\mathcal I}(x)| \leq 2$.  Because $B_{R_{\mathcal I}(x)}(x)$ actually belongs to $\mathcal I$, it is also an immediate consequence that both \eqref{highderiv} and \eqref{firstderiv} will hold.  Thus we have established in a very explicit way that the scale function can simply be taken to measure the largest scale on which \eqref{highderiv} and \eqref{firstderiv}.  If the phase $f$ is finite type in the sense mentioned at the end of section \ref{ftsec} (which will happen when $f$ is polynomial or, in the Carnot-Carath\'{e}odory context, when $f$ is annihilated by applying any sufficiently long sequence of the distinguished vector fields), then the scale function will simply measure the largest scale on which the magnitude of the derivative is constant.

\subsection{Quantitative results for convex phases}
\label{bnwsec}

In this section we establish several results relating theorem \ref{maintheorem} to the earlier theorem of Bruna, Nagel, and Wainger.  Specifically we consider the question of establishing uniform radial tameness for an appropriate phase or phases and then give the proof of corollary \ref{bnwcor}.
We begin with two propositions which establish uniform radial tameness for first polynomial convex phases and then for convex phases of finite type (in the same sense meant by Bruna, Nagel, and Wainger).

\begin{proposition}
Suppose $f$ is a convex function on a convex open set $\Omega \subset \R^d$ containing the origin and that $f$ and its gradient vanish at the origin.  If $f$ is a polynomial then it is uniformly radially tame to order $m$ with a constant depending only on the degree and $m$.
\end{proposition}
\begin{proof}
It suffices to restrict attention to a single ray beginning at the origin.  Suppose $f(t)$ is a convex polynomial on $[0,T]$ with $f(0) = f'(0) = 0$.  Rescaling \eqref{bigfinish} to the interval $(0,t_0)$, we will have that
\begin{equation} t_0^{k} \sup_{0 < t < t_0} |f^{(k)}(t)| \lesssim \sup_{0 < t < t_0} |f(t)| \label{build} \end{equation}
for any $k$ with implicit constant depending only on $k$ and the degree of $f$.  Convexity implies for all $t$ that
\[ f'(t) \geq t^{-1} f(t) \]
so we conclude from \eqref{bigfinish} and monotonicity of $f$ that
\[ t_0 f'(t_0) \approx f(t_0) \]
for all $t_0 \in (0,T)$, and so we can additionally conclude that
\[ |f(t_0)|^{k-1} |f^{(k)}(t_0)| \lesssim |f'(t_0)|^k \]
for any $t_0 \in (0,T)$ with implicit constant depending only on the degree of $f$ and on $k$.
\end{proof}

\begin{proposition}
Let $f$ be a smooth convex function on convex open set $\Omega \subset \R^d$.  Suppose that $f$ is finite type in the sense of Bruna, Nagel, and Wainger, namely, that every tangent line to $f$ has only finite order of contact.  For $x_0$ belonging to any compact convex subset $\Omega' \subset \Omega$, the convex functions
\[ f_{x_0}(x) := f(x) - f(x_0) - (x-x_0) \cdot \nabla f(x_0) \]
are uniformly radially tame on $\Omega'$ to any finite order $m$ with respect to the origin point $x_0$, and the constant is bounded over all $x_0 \in \Omega'$.
\end{proposition}
\begin{proof}
The finite type condition guarantees that for every pair of a point $x_0$ and unit vector $\omega$, there is a finite $k \geq 2$ such that
\[ \sum_{i=2}^k |(\omega \cdot \nabla)^i f(x_0)| \neq 0. \]
Since this sum (for fixed $k$) is a continuous function of $x_0$ and $\omega$, we may assume by compactness that there is a single $k$ such that
\[ \sum_{i=2}^k |(\omega \cdot \nabla^i f(x_0)| \geq C_{\Omega'} > 0 \]
for any pair $(x_0, \omega) \in \Omega' \times {\mathbb S}^{d-1}$.  Now we can conclude that
\[  \sup_{1 \leq i \leq k} \sup_{0 < t < 1}  |x-x_0|^{i} \left| \frac{d^i}{dt^i} f( x_0 + t(x-x_0)) \right| \geq C_{\Omega'} \]
whenever $x_0, x$ belong to $\Omega'$.  Now
\begin{align*}
\sup_{0 < t < 1}  |x-x_0|^{m} & \left| \frac{d^m}{dt^m} f( x_0 + t(x-x_0)) \right| \lesssim ||f||_{C^m(\Omega')} \\ & \lesssim \frac{||f||_{C^m(\Omega')}}{C_{\Omega'}} \sup_{1 \leq i \leq k} \sup_{0 < t < 1}  |x-x_0|^{i} \left| \frac{d^i}{dt^i} f( x_0 + t(x-x_0)) \right| . \end{align*}
This inequality guarantees that \eqref{sharpcond} will hold uniformly on $(0,1)$ for the functions $t \mapsto f(x_0 + t(x-x_0)) - f(x_0) - t (x-x_0) \cdot \nabla f(x_0)$ with $x,x_0 \in \Omega'$ provided that $\epsilon$ is chosen to be a suitably small constant multiple of $|x-x_0|$.  Thus we may conclude
\[ \sup_{0 < t < 1}  |x-x_0|^{k}  \left| \frac{d^k}{dt^k} f_{x_0}( x_0 + t(x-x_0)) \right| \lesssim |f(x) - f(x_0) - (x-x_0) \cdot \nabla f(x_0)| \]
uniformly for $x,x_0 \in \Omega'$ for any fixed choice of $k$.  From here the rest of the proof follows exactly as it did in the previous proposition.
\end{proof}

Now we can take up the proof of corollary 1.  As a reminder, for uniformly radially tame $f$, we seek to establish the inequality
\begin{equation*} \left| \int_{\Omega} e^{i \lambda f(x)} (f(x))^\ell \psi(x) dx \right| \lesssim \int_{\Omega} \frac{|f(x)|^\ell \sum_{k=0}^{m-1} |(x \cdot \nabla)^k \psi(x)|  }{(1 +  |\lambda f(x)|)^{m-1}} dx  \end{equation*}
with implicit constant depending only on $m$, $d$, the nonnegative integer $\ell$, and the constant of uniform radial tameness.  Under the circumstances, it suffices to assume $\lambda = 1$ since uniform radial tameness is invariant under scalar multiplication of $f$ as are the condtions \eqref{lipschitz} through \eqref{firstderiv}.  Following the proof, we record how to establish the inequality
\begin{equation} \int_{\Omega} \frac{|f(x)|^\ell}{(1 +  |\lambda f(x)|)^{\ell+d+1}} dx \leq C_{d} |\lambda|^{-\ell} \left| \set{ x \in \Omega}{ |f(x)| < |\lambda|^{-1}} \right| \label{basicsub} \end{equation}
for our convex phase $f$ (where $d$ equals the dimension: $\Omega \subset \R^d$), which brings \eqref{sublev} in line with the results of Bruna, Nagel, and Wainger (when we require $m \geq d+\ell + 2$).

\begin{proof}[Proof of corollary \ref{bnwcor}.]
In this case, we apply the machinery of theorem \ref{maintheorem} when $\Omega$ is equal to $\R^d$.  We will let the indices $j$ belong to $\Z$ and define
\begin{equation} \Phi_{j,x}(t) := e^{3^{j} t} x. \label{convexdef}
\end{equation}
for $x \neq 0$ and $\Phi_{j,0}(0) = 0$ for all $j$; in other words, the balls $B_j(x)$ are intervals in the ray from the origin through $x$ when $x \neq 0$ and are simply points when $x = 0$.  These rays and the point $\{0\}$ are exactly the leaves.  The basic conditions of theorem \ref{maintheorem} are easily checked (the homogeneous space structure on the real line given by dyadic intervals centered at $x$ is well-known, and this present construction is only a trivial variation).  In this case, regularity of measure is established by means of the polar coordinates formula:
\[ \int f = \int_{{\mathbb S}^{d-1}} \left[ \int_0^\infty f(r \omega) r^{d-1} dr \right] d \sigma(\omega) \]
(so we take the measure on the point $\{0\}$ to be zero).  In particular, the smoothness conditions on the Jacobian will hold as long as we restrict $j \leq C$ at every point for some fixed $C$ (which is an acceptable restriction from the point of view of the compatibility condition).  

Now suppose that the phase $f$ is uniformly radially tame on the domain $\Omega \subset \R^d$ containing the origin.   Convexity implies the inequalities 
\begin{equation} f(x) \leq x \cdot \nabla f(x)\label{convexity} \end{equation}
for all $x \in \Omega$.   Momentarily fix attention on a single ray emanating from the origin (and if $f$ is not strictly convex on this ray, assume that we are far enough from the origin that $f \neq 0$).  Now assuming that $f$ is uniformly radially tame to order $m \geq 2$ with constant $C$, we have
\begin{align*}
 \left| (x \cdot \nabla) \left( \frac{f(x)}{x \cdot \nabla f(x)} \right)  \right| & =  \left| 1 - \frac{f(x) (x \cdot \nabla )^2 f(x)}{( x \cdot \nabla f(x))^2} \right| \\
 &  =  \left| 1 - \frac{ f(x) \left[ |x|^2 \frac{d^2}{dr^2} f(x) + |x| \frac{d}{dr} f(x) \right]}{ |x|^2 (\frac{d}{dr} f(x))^2} \right| \\
 & = \left| 1 - \frac{f(x) \frac{d^2}{dr^2} f(x)}{(\frac{d}{dr} f(x))^2} - \frac{f(x)}{x \cdot \nabla f(x)} \right| \leq C.
 \end{align*}
Integrating along rays will give that
\[ \left| \frac{f( \rho x)}{ \rho x \cdot \nabla f( \rho x)} - \frac{f(x)}{x \cdot \nabla f(x)} \right| \leq C \ln \rho \]
for any $\rho > 1$ (with a similar inequality when $\rho < 1$).
We will define $R(x)$ (when $\nabla f(x) \neq 0$) to be the largest integer $j$ such that
\[ 3^j \leq \frac{1}{4C} \frac{f(x)}{x \cdot \nabla f(x)}. \]
By \eqref{convexity}, this upper bound is at most $\frac{1}{4C}$, so $R(x)$ is uniformly bounded above for each $x$.   Fix an $x$ and take $\rho := e^{3^{R(x)} t}$ for $-1 < t < 1$.  We conclude that
\begin{equation} \left| \frac{f( \rho x)}{ \rho x \cdot \nabla f( \rho x)} - \frac{f(x)}{x \cdot \nabla f(x)} \right| \leq C 3^{R(x)} \leq \frac{1}{4} \frac{f(x)}{x \cdot \nabla f(x)}, \label{establishcomp} \end{equation}
from which we conclude that $R(x)$ and $R(y)$ should differ by at most $1$ when $y$ belongs to $B_{R(x)}(x)$.  The hypothesis \eqref{lipschitz} now follows immediately with $N = 2$ by the triangle inequality (since $R(x)$ and $R(z)$ both differ by at most one from $R(z)$ when $z$ belongs to the intersection).

Next we compute:
\begin{align*}
|d^k_x f|_{R(x)}  := & \sup_{1 \leq k' \leq k} \left|  \frac{d^{k'}}{dt^{k'}}  \left.  f( e^{3^{R(x)} t} x) \right|_{t=0} \right| = \sup_{1 \leq k' \leq k} \left| 3^{k' R(x)} (x \cdot \nabla)^{k'} f(x) \right| \\
 \approx & \sup_{1 \leq k' \leq k} \left| \frac{(f(x))^{k'}}{(x \cdot \nabla f(x))^{k'}} (x \cdot \nabla)^{k'} f(x) \right|.
\end{align*}
By virtue of the assumption of uniform radial tameness to order $m$ (and the Leibniz rule), we can therefore conclude that $|d^k_x f|_{R(x)} \approx f(x)$ provided $k \leq m$.  The condition \eqref{highderiv} follows immediately.  
Finally, since \eqref{establishcomp} implies that
\[ \frac{3}{4} \frac{f(x)}{x \cdot \nabla f(x)} \leq \frac{f(y)}{ x \cdot \nabla f(y)} \leq \frac{5}{4} \frac{f(x)}{x \cdot \nabla f(x)} \]
for any $y \in B_{R(x)}(x)$,  it must be the case that
\[ \rho \frac{d}{d \rho} \ln f(\rho x) \leq \frac{4}{3} \frac{x \cdot \nabla f(x)}{f(x)} \]
for any $\rho$ with $\rho x \in B_{R(x)}(x)$, so we again integrate with respect to the variable $\rho$ to conclude that
\[ \left| \ln f(\rho x) - \ln f(x) \right| \leq \frac{4}{3} \frac{x \cdot \nabla f(x)}{f(x)} \left| \ln \rho \right| \leq \frac{1}{3C}. \]
Since we know $|d_y f|_{R(y)} \approx f(y)$, we have thus established the final hypothesis \eqref{firstderiv}.

Having established \eqref{lipschitz} through \eqref{firstderiv}, we may apply the conclusion \eqref{mainineq}:
\[  \left| \int_{\Omega} e^{i \lambda f(x)} (f(x))^{\ell} \psi(x) dx \right| \lesssim \int_{\Omega} \frac{\sum_{k=0}^{m-1} |d_x^k f^\ell \psi|_{R(x)}}{(1 +  \lambda f(x))^{m-1}} dx \]
(where we have already simplified the denominator since $|d_x f|_{R(x)} \approx f(x)$).  By the Leibniz rule and our estimates for $|d^k_x f|_{R(x)}$ we will have
\[ \sum_{k=0}^{m-1} |d_x^k f^\ell \psi|_{R(x)} \lesssim (f(x))^{\ell} \sum_{k=0}^{m-1} |d_x^k \psi|_{R(x)} \]
(with constant depending on $m$ and $\ell$).
The only remaining modification is that
\[ \sum_{k=0}^{m-1} |d_x^k \psi|_{R(x)} \lesssim \sum_{k=0}^{m-1} |(x \cdot \nabla)^k \psi(x)| \]
since $R(x)$ is bounded uniformly above.
\end{proof}

Lastly we turn our attention to \eqref{basicsub}. Assume $\lambda > 0$ and let \begin{align*}  F_0 & := \set{x \in \Omega}{f(x) < \lambda^{-1}}, \\ F_k & := \set{x \in \Omega}{2^{k-1} \lambda^{-1} \leq f(x) < 2^{k} \lambda^{-1}}, \ k > 0. \end{align*}
  We have
\[ \int_{\Omega} \frac{(f(x))^{\ell}}{(1 + \lambda f(x))^{d+\ell+1}} dx \leq \lambda^{-\ell} \sum_{k=0}^{\infty} \int_{F_k} \frac{1}{(1 + \lambda f(x))^{d+1}} dx. \]
Convexity of $f$ (and $f(0), \nabla f(0) = 0$) implies that $f(\alpha x) \geq \alpha f(x)$ when $\alpha \geq 1$.  In particular this means that $F_k \subset 2^k F_0$ for each $k$.  Consequently
\[\sum_{k=0}^{\infty} \int_{F_k} \frac{1}{(1 + \lambda f(x))^{d+1}} dx \leq  |F_0| + \sum_{k=1}^\infty (2^{-(k-1)})^{d+1} 2^{dk} |F_0|. \] 
This geometric series converges and yields a finite constant for \eqref{basicsub}.

\subsection{Carnot-Carath\'{e}odory}

In this section we recall the framework of Street \cite{street2011} and several of the results proved there.  The main purpose of doing so is to establish theorem \ref{ccworks}, which exploits these results to show that Street's Frobenius theorem produces a family of balls satisfying compatibility, engulfing, weak doubling, smooth nesting, and smooth engulfing (properties (i) through (v) from the introduction).  Moreover, it establishes that the definition of leaves from the introduction coincides with leaves in the Frobenius theorem.  Finally, it provides an estimate for the smoothness of the Jacobian when integrating Lebesgue measure on a leaf.  This does not explicitly prove the regularity of measure hypothesis, since we will need a global measure $\mu$ which factors as Lebesgue measure (up to smooth density) onto the leaves.  It appears that such a measure may not exist in certain exceptionally pathological cases.  However, if the dimension of the leaves is constant on some neighborhood, then the classical coarea formula combines with the Jacobian estimate in theorem \ref{ccworks} to guarantee that the $d$-dimensional Lebesgue measure factors locally in exactly the way required by regularity of measure.  Beyond this, was already seen in the previous section (near the origin), it is often possible to establish regularity of measure directly even when the dimension of the leaves is not constant.

We now recall the framework of Street \cite{street2011}.  Begin with a finite collection of $C^1$ vector fields $X_1,\ldots,X_q$ on $\Omega$.  Each vector field has a nonzero ``formal degree'' $d_1,\ldots,d_q$ belonging to $[0,\infty)^d$.  Fix some compact subset $K \subset \Omega$ and $\xi \in (0,1]^d$.  Suppose that for any $x_0 \in K$ and any $a = (a_1,\ldots,a_q) \in (L^\infty([0,1]))^q$ with $|| \sqrt{ \sum_{i=1}^q |a_i|^2}||_{L^\infty} < 1$,  the ODE
\[ \gamma'(t) = \sum_{i=1}^q \xi^{d_i}  a_i(t) X_i(\gamma(t)) \]
has a (weak) solution on $[0,1]$ with initial condition $\gamma(0) = x_0$.

Next fix a set $A \subset \set{ \delta \in [0,1]^{d}}{\delta \neq 0, \ \delta \leq \xi}$.  This set $A$ represents the allowable multiscales $\delta$ (where ``allowable'' is in principle determined by the context in which the Carnot-Carath\'{e}odory machinery is being applied).  The present purposes, we add a constraint to the collection $A$ of allowable $\delta$:  we assume that $\delta = (\delta_1,\ldots,\delta_d) \in A$ implies $(\epsilon \delta_1,\ldots, \epsilon \delta_d) \in A$ for any $\epsilon \in (0,1)$.  Note that this isotropic dilation condition on $A$ will generally be true in applications of interest; in particular it holds in the case of ``weak comparability.''  The Carnot-Carath\'{e}odory ball $B_{(X,d)}(x_0,\delta)$ is defined as the set $y \in \Omega$ such that there exists an $a \in (L^\infty([0,1]))^q$ with $|| \sqrt{ \sum_{i=1}^q |a_i|^2}||_{L^\infty} < 1$ as before such that the necessarily unique solution of
\[ \gamma'(t) = \sum_{i=1}^q \delta^{d_i}  a_i(t) X_i(\gamma(t)) \]
with $\gamma(0) = x_0$ has $\gamma(1) = y$.

The assumptions begin with the integrability condition: for every $\delta \in A$ and $x \in K$, it must be the case that
\[ [ \delta^{d_i} X_i , \delta^{d_{i'}} X_{i'}] = \sum_{k=1}^q c^{k,\delta,x}_{i,i'} \delta^{d_k} X_k \]
at every point $y \in B_{(X,d)}(x,\delta)$.  Next assume that for some $m \geq 2$, the vector fields are $C^m$ on $B_{(X,d)}(x,\xi)$ for every $x \in K$, that $X^{\alpha} c_{i,i'}^{k,\delta,x}$ is continuous on this same ball whenever $|\alpha| \leq m$, and 
\begin{align} \sup_{x \in K} ||X_i||_{C^m(B_{(X,d)}(x,\xi))} & < \infty, \label{admissible1} \\
\mathop{\sup_{\delta \in A}}_{x \in K} \sum_{|\alpha| \leq m} || (\delta X)^\alpha c_{i,i'}^{k,\delta,x}||_{C^0(B_{(X,d)}(x,\xi))} & < \infty. \label{admissible2}
\end{align}
(Note that the norms are taken with respect to some implicit, fixed coordinate system on $\Omega$.)
For each $x \in K$, let $n_0(x)$ be the dimension of the span of $X_1,\ldots,X_q$ at $x$.  For each $\delta \in A$, Street (in agreement with Nagel, Stein, and Wainger) identifies an appropriate subcollection $J(x,\delta) \subset \{1,\ldots,q\}$ and defines a mapping on a neighborhood of the origin in $\R^{J(x,\delta)}$
\[ \Phi_{x,\delta}(u) := \exp \left( \sum_{i \in J(x,\delta)} u_i \delta^{d_i} X_{i} \right)x. \]

\begin{theoremst}[Street \cite{street2011}]
There are $m$-admissible constants $\rho$, $r_2 < r_1$ such that the following hold for all $\delta \in A$ and $x \in K$:
\begin{itemize}
\item $B_{(X,d)}(x, \rho \delta) \subset \Phi_{x,\delta}( \B^{n_0(x)}(r_2)) \subset \Phi_{x,\delta}(\B^{n_0(x)}(r_1)) \subset B_{(X,d)}(x,\delta)$.
\item $\Phi_{x,\delta}(u)$ is one-to-one on $\B^{n_0(x)}(r_1)$
\item If $Y_i$ is the pullback of $\delta^{d_i} X_i$ under the map $\Phi_{x,\delta}$ on the ball $\B^{n_0(x)}(r_1)$, then $||Y_i||_{C^m} \lesssim 1$.  Furthermore, 
\[ (Y_{J(x,\delta)_1},\ldots,Y_{J(x,\delta)_{n_0(x)}}) = (I + B(x,u)) \nabla_u \]
for some $C^m$ matrix $B(x,u)$ of norm at most $\frac{1}{2}$.
\item For all $u \in B^{n_0(x)}(r_1), |\det_{n_0(x) \times n_0(x)} d \Phi_{x,\delta}(u)| \approx |\det_{n_0(x) \times n_0(x)} \delta X(x)|$.
\end{itemize}
Here $\delta X$ is the matrix with columns $\delta^{d_i} X_i$ and $\det_{k \times k}$ is the vector whose entries are the determinants of all $k \times k$ minors of that matrix. An $m$-admissible constant is one which depends only on upper bounds of \eqref{admissible1}, \eqref{admissible2}, the dimension of $\Omega$, $q$, $d$, lower bounds for the coordinates of $\xi$, and upper and lower bounds for the coordinates of $\sum d := (\sum_{i=1}^d d_1^i,\ldots, \sum_{i=1}^d d_{q}^i)$. \label{streetheorem}
\end{theoremst}

We now come to the final theorem, which is essentially a repackaging of a number of Street's definitions and estimates to illustrate that he implicitly constructed a space of exactly the sort we have defined in the present work:
\begin{theorem} \label{ccworks}
Under the same hypotheses outlined above for theorem 6.4 of \cite{street2011}, define balls $B_j(x) := \Phi_{x,M^j}(\B^{n_0(x)}(r_1))$ and homeomorphisms $\Phi_{j,x}(u) := \Phi_{x,M^j}(r_1^{-1} u)$ when $M^j \in A$.  There is a choice of constant $M$ such that the hypotheses (i) through (v) of theorem \ref{maintheorem} are satsified when the indices $j$ are restricted to have $M^j \in A$ and each component of $j$ sufficiently negative.  Furthermore, the leaves of the foliation given by the Frobenius theorem are leaves in the sense of theorem \ref{maintheorem}, and 
\[ \frac{1}{\mathrm{Vol}(B_j(x))} \int_{B_j(x)} f d \mu_L = \int_{\B} f \circ \Phi_{j,x}(t) J_{j,x}(t) dt \]
when $\mu_L$ is the induced Lebesgue measure on the leaf $L$ for some nonnegative function $J_{j,x}$ with $J_{j,x} \approx 1$ and $||J_{j,x}||_{C^m} \lesssim 1$.
\end{theorem}

\begin{proof}
Let $M$ be some constant greater than one to be determined momentarily.  For suitable $j \in \Z^d$ we define $B_j(x) := \Phi_{x,M^j}(\B^{n_0(x)}(r_1))$ and $\Phi_{j,x}(u) := \Phi_{x,M^j}(r_1^{-1} u)$.  We will first show that this system satisfies the axioms (i) through (vi).  In this case, the compatibility condition (i) is trivially satisfied because the set $A$ of admissible $\delta$ has not been taken to depend on $x$ and it has explicitly been assumed to be closed under contractions $\delta \mapsto M^{-1} \delta$.  It is essentially a matter of definitions to show that when $\delta' \leq \delta$ 
\[ B_{(X,d)}(y,\delta') \cap B_{(X,d)}(x,\delta) \neq \emptyset \implies B_{(X,d)}(y,\delta') \subset B_{(X,d)}(x, 2^p \delta) \]
for any $p$ with $p \sum_{i=1}^d d_k^i \geq 1$ for all $k=1,\ldots,q$.  This follows easily from concatenating paths and rescaling; note that $p$ is an admissible constant under Street's terminology.  As long as $M^{-1} \leq 2^{-p} \rho$ we have
\begin{align*}
 B_j(x) & \cap B_{j'}(y)  \neq \emptyset \Rightarrow B_{(X,d)}(x,M^j) \cap B_{(X,d)}(y,M^{j'}) \neq \emptyset \\
 \Rightarrow & B_{j'}(y) \subset B_{(X,d)}(y,M^{j'}) \subset B_{(X,d)}(x,2^p M^j) \\ & \hspace{40pt} \subset B_{(X,d)}(x, \rho M^{j+1}) \subset B_{j+1}(x).
\end{align*}

To establish weak doubling, we use the doubling condition from corollary 6.4 of \cite{street2011}:
\[ \mathrm{Vol} ( B_{(X,d)}(x,2 \delta)) \lesssim \mathrm{Vol}(B_{(X,d)}(x,\delta)) \]
for $\delta$ sufficiently small, where $\mathrm{Vol}$ represents the $n_0(x)$-dimensional Hausdorff measure.  We know explicitly from Street's paper that the measure of such a ball is never zero or infinity when the entries of $\delta$ are nonzero.  It is straightforward to see that the nesting property guarantees that the doubling property holds for the dyadic balls $B_j(x)$ as well.  Now suppose $B_{j}(x_1),\ldots, B_j(x_N)$ are mutually disjoint and that some ball $B_{j+1}(x)$ intersects all of the balls $B_{j+1}(x_k)$.  Then we have the containments $B_j(x) \subset B_{j+2}(x_k)$ and $B_{j}(x_k) \subset B_{j+2}(x)$ for each $k$.  But now the observations 
\begin{align*}
 \sum_{k=1}^N \mathrm{Vol}(B_j(x_k)) & \leq \mathrm{Vol}(B_{j+2}(x)), \\
\mathrm{Vol}(B_{j+2}(x)) & \approx \mathrm{Vol}(B_{j}(x)) \leq \mathrm{Vol}(B_{j+2}(x_{k}))  \approx \mathrm{Vol}(B_{j}(x_k))
\end{align*}
combine to give the uniform inequality
\[ N \mathrm{Vol}(B_{j+2}(x)) \lesssim \mathrm{Vol}(B_{j+2}(x)) \]
which, in turn, gives a uniform upper bound on $N$ because the volume is known to be nonzero.

Now on to the smooth structures:  the smooth nesting property follows immediately from theorem 6.4 with $c := r_2/r_1$.  Smooth engulfing is almost equally immediate.  Given two balls $B_{j}(x)$ and $B_{j'}(x')$ of comparable scale with a nontrivial intersection, there will be a third ball $B_{j''}(x'')$ of another comparable scale that contains them both.
Now
\[ \Phi_{x,\delta}^{-1} \circ \Phi_{x',\delta'} = ( \Phi^{-1}_{x'',\delta''} \circ \Phi_{x,\delta})^{-1} \circ ( \Phi^{-1}_{x'',\delta''} \circ \Phi_{x',\delta'}). \]
On this third ball, the map we have by pullbacks that
\[ \Phi_{x,\delta}^{-1} \circ \Phi_{x',\delta'} (u) = \exp \left( \sum_{i \in J(x',\delta')} u_i (\delta')^{d_i} (\delta'')^{-d_i} Y_i'' \right) \Phi_{x,\delta}^{-1}(x'). \]
We know that the pullback vector fields $Y_i''$ of $(\delta'')^{d_i} X_i$ via $\Phi_{x'',\delta''}$ are uniformly in $C^m$, so the mapping $ \Phi_{x,\delta}^{-1} \circ \Phi_{x',\delta'} (u)$ must be uniformly $C^m$ as well.  Moreover, if we choose $x'' = x$, we will have that $( \Phi^{-1}_{x'',\delta''} \circ \Phi_{x,\delta})^{-1}$ will also be uniformly in $C^m$ because of comparability of Jacobians
\begin{align*}
 |\det_{n_0(x) \times n_0(x)} d \Phi_{x,\delta}(u)| \approx & |\det_{n_0(x) \times n_0(x)} \delta X (x)| \approx |\det_{n_0(x) \times n_0(x)} \delta'' X (x)| \\ & \approx |\det_{n_0(x) \times n_0(x)} d \Phi_{x,\delta''}(u'')|; \end{align*}
these imply that the Jacobian determinant of $\Phi^{-1}_{x'',\delta''} \circ \Phi_{x,\delta}$ is uniformly bounded above and below, and Cramer's rule then implies that the inverse mapping will be uniformly in $C^m$ depending on the $C^m$ norm of the mapping $\Phi^{-1}_{x'',\delta''} \circ \Phi_{x,\delta}$ itself.

Finally, regarding the foliations, leaves, and measures, see appendices B and C of \cite{street2011}.  In particular, we have the formula
\[ \int_{B_j(x)} f d \mu_L = \int f \circ \Phi_{j,x}(t) |\det_{n_0(x) \times n_0(x)} (d \Phi(t))| dt \]
where $| \cdot |$ is the usual Euclidean length. The magnitude of $|\det_{n_0(x) \times n_0(x)} (d \Phi(t))|$ is shown by Street to be uniformly bounded above and below by the volume of the ball.  To show smoothness, we exploit lemma 4.16 and proposition 4.17 of \cite{street2011}, which together show that
\[ \frac{|\det_{n_0(x) \times n_0(x)} d \Phi_{x,\delta}(u)|}{ |\det_{n_0(x) \times n_0(x)} \delta X(\Phi_{x,\delta}(u))|} \]
is uniformly in $C^m$ on the ball.  Since the pullback vector fields $Y_i$ are uniformly in $C^m$ and uniformly span (meaning that there is an $n_0(x)$-tuple which when grouped into a matrix are uniformly close to the identity matrix, and in particular, have determinant uniformly bounded below), it suffices to show that
\[ \left| (\delta X)_{i_1} \cdots (\delta X)_{i_k} \det_{n_0(x) \times n_0(x)} (\delta X)(x) \right| \lesssim \left| \det_{n_0(x) \times n_0(x)} (\delta X)(x) \right| \]
uniformly for any choice of $i_1,\ldots,i_k$ with $k \leq m$.  From the proof of lemma 4.6 in \cite{street2011}, we see that we may write
\[ (\delta X)_{i_k} \det_{n_0(x) \times n_0(x)} (\delta X)(x) \] as some smooth matrix with admissible norm times $\det_{n_0(x) \times n_0(x)} (\delta X)(x)$ itself (admissible because the smooth functions appearing are literally those bounded by \eqref{admissible1} and \eqref{admissible2}).  The final result follows by induction on $k$.
\end{proof}

%Let $K_0$ be the subset of $K$ on which $n_0(x) = 0$.  Fix any point $x \in K \setminus K_0$; there must be some open neighborhood of $x$ on which some vector field $X_i$ is nonvanishing.  After making a $C^m$ change-of-variables (which will introduce a Jacobian factor of smoothness $C^{m-1}$, we may assume that this vector field $X_i$ coincides with the $d$-th coordinate direction.  It will also be the case that for any $i' \neq i$, there will be a function $c_{i'}$ (with the same degree of smoothness as the vector fields themselves, now $C^{m-1}$) such that
%\[ X_{i'} - c_{i'} \frac{\partial}{\partial x_d} \]
%annihilates $x_d$ for each $i'$.  Let $U$ be a small neighborhood of $x$ which has the form $U' \times I$ for some interval $I$.  Clearly on $U$, every leaf $L$ will have the form $L' \times I$, where $L'$ is itself a leaf of the foliation on $(x_1,\ldots,x_{d-1})$ generated by $X_{i'} - c_{i'} \frac{\partial}{\partial x_d}$ (with the final coordinate held constant).  Note that we know this new system of vector fields will still be involutive when restricted to the level sets of $x_d$ because each of the vector fields is tangent to these level sets.

\bibliography{mybib}

\end{document}